\documentclass[12pt]{article}%
\usepackage{amsmath}
\usepackage{amscd}
\usepackage{amsfonts}
\usepackage{amssymb}%
\newtheorem{theorem}{Theorem}[section]
\newtheorem{lemma}[theorem]{Lemma}
\newtheorem{proposition}[theorem]{Proposition}

\newtheorem{definition}[theorem]{Definition}
\newtheorem{corollary}[theorem]{Corollary}

\newenvironment{proof}{\bf Proof. \rm}{$\Box$}
\newcommand{\be}{\begin{equation}}
\newcommand{\ee}{\end{equation}}
\newcommand{\bes}{\begin{equation*}}
\newcommand{\ees}{\end{equation*}}

\newcommand{\cS}{\mathcal{S}}

\newcommand{\Rp}{\mathbb{R}_+}

\begin{document}


\title{Dilation theorems for contractive semigroups}

\author{Orr Shalit}
\date{Written in June, 2007, slightly revised on April, 2010}
\maketitle


\section{Introduction}
This note records some dilation theorems about contraction semigroups on a Hilbert space - all of which fall into the categories ``known" or ``probably known" - that I proved while working on my PhD in mathematics (under the supervision of Baruch Solel). It is convenient to have them recorded for reference.

In section 2, we prove that every continuous two parameter semigroup of contractions on a Hilbert space
admits a minimal isometric dilation. The main idea of the proof - to use Ando's theorem for the pair of cogenerators of the commuting one-parameter semigroups - was suggested to me by Paul Muhly. This was already proved twice before, by M. S\l oci\'nski and by M. Ptak.

In section 3 we prove that certain multiparameter semigroups have regular unitary and isometric dilations. These are semigroups of operators indexed by semigroups $\cS$ of the from
$$\cS = \sum_{i}\cS_i ,$$
where $\cS_i$ are \emph{commensurable subsemigroups} of $\Rp$ (see definitions below).

In section 4 we prove that a ``multiparameter" semigroup of
coisometries on a Hilbert space, not necessarily continuous, has a
unitary and an isometric dilation.

\section{A continuous parameter version of Ando's Theorem}
\begin{theorem}\label{thm:mainthm}
Let $\{T_1 (s)\}_{s\geq0}$ and $\{T_2 (s)\}_{s\geq0}$ be two commuting continuous semigroups of contractions on a complex Hilbert space $H$. There exists a complex Hilbert space $K \supseteq H$ and a pair of two commuting continuous semigroups of isometries $\{V_1 (s) \}_{s\geq0}$ and $\{V_2 (s) \}_{s\geq0}$ such that for all $s,t \geq0$,
\begin{equation}\label{eq:what we want}
T_1 (s) T_2 (t) = P_H V_1 (s) V_2 (t) \big|_H ,
\end{equation}
where $P_H$ denotes the orthogonal projection on $H$. $K$, $\{V_1 (s) \}_{s\geq0}$ and $\{V_2 (s) \}_{s\geq0}$ can be chosen to be \emph{minimal}, in the sense that
\begin{equation}
K = \bigvee_{s,t\geq0} V_1 (s) V_2 (t) H .
\end{equation}
\end{theorem}
By \emph{continuous semigroup} we mean strongly continuous (or
weakly continuous at $0$, which amounts to the same). I have been
looking for this result for a while. Paul Muhly (on 5.6.2006 in
Baruch Solel's office) suggested to prove this by applying Ando's
theorem to the cogenerators of $\{T_1 (s)\}_{s\geq0}$ and $\{T_2
(s)\}_{s\geq0}$. The purpose of this note is to work through the
details of this proof. A few days after proving this result, I was
told by Baruch Solel that this result was already proved in 1974
by S\l oci\'nski (\cite{Slocinski}), and later, in 1985, and by a
different method, by Ptak (\cite{Ptak}), under the assumption that
$H$ is seperable. The proof we give here is close to S\l
oci\'nski's. The difference is that in order to justify the fact
that the dilations of the cogenerators serve as cogenerators
themselves, S\l oci\'nski uses the technology of spectral
measures, whereas I use the following lemma.

\begin{lemma}\label{lem:evalnot1}
Let $T_1$ and $T_2$ be two contractions on a Hilbert space $H$, such that neither of them has $1$ as an eigenvalue. Then there exists an isometric dilation  $(G,U_1,U_2)$ of $(H,T_1,T_2)$ such that neither $U_1$ nor $U_2$ has $1$ as an eigenvalue.
\end{lemma}
\begin{proof}
Let $(K,V_1,V_2)$ be the isometric dilation described in the proof of Theorem I.6.1, \cite{SzNF70} (Ando's Theorem). For $i = 1,2$, denote
$$\tilde{L}_i = \{x \in K : V_i x = x \}.$$
For all $x \in \tilde{L}_i$, we can write $x = x_1 + x_2$, with $x_1 \in H$ and $x_2 \in H^\perp$. It follows from the properties of that dilation that $V_i x = T_i x_1 + \tilde{x}_2$, where $\tilde{x}_2 \in H^\perp$, thus for $i = 1,2$ we see that $\tilde{L}_i \subseteq H^\perp$ (because $T_i$ was assumed to have no invariant vectors). Thus $K = H \oplus M \oplus L_1 \oplus L_2$, where $L_1 = \tilde{L}_1$, $L_2 = \left( \tilde{L}_1 \vee \tilde{L}_2 \right) \ominus \tilde{L}_1$, and $M = K \ominus \left( H \oplus L_1 \oplus L_2 \right)$. Now if $x \in \tilde{L}_2$, then $V_2 V_1 x = V_1 V_2 x = V_1 x$, whence $V_1 \tilde{L}_2 \subseteq \tilde{L}_2$, and this also means that $V_1 L_2 \subseteq L_1 \oplus L_2$. It follows that $V_1$ has the following structure:
\begin{equation}
V_1 = \begin{pmatrix}
  A &  0 & 0   & \\
  B &  I & C   & \\
  D &  0 & W_1 & \\
\end{pmatrix} .
\end{equation}
(The first row corresponds to $H \oplus M$, the second and third correspond to $L_1$ and $L_2$, respectively). Since $V_1$ is an isometry, we have that
$$I = V_1^* V_1 = \begin{pmatrix}
  *   &  B^* &* & \\
  *   &  *   & C& \\
  *   &  *   & * & \\
\end{pmatrix} ,$$
so we conclude that $B = C = 0$ \footnote{Of course, these operators are not really equal because they operate between different spaces.}, and $V_1$ turns out to have the form
\begin{equation}
V_1 = \begin{pmatrix}
  A &  0 & 0   & \\
  0 &  I & 0   & \\
  D &  0 & W_1 & \\
\end{pmatrix} .
\end{equation}
Now, $V_2 L_1 \subseteq L_1$, $V_2 L_2 \subseteq L_1 \oplus L_2$, so $V_2$ has the form
\begin{equation}
V_2 = \begin{pmatrix}
  X   &  0   & 0 & \\
  Y   &  W_2 & Z & \\
  R   &  0   & T & \\
\end{pmatrix} .
\end{equation}
We claim that $Y = Z = 0$. Indeed,
$$
V_1 V_2 = \begin{pmatrix}
  A &  0 & 0   & \\
  0 &  I & 0   & \\
  D &  0 & W_1 & \\
\end{pmatrix}
\begin{pmatrix}
  X   &  0   & 0 & \\
  Y   &  W_2 & Z & \\
  R   &  0   & T & \\
\end{pmatrix} =
\begin{pmatrix}
  AX         &  0     & 0    & \\
  Y          &  W_2   & Z    & \\
  DX + W_1 R &  0     & W_1 T& \\
\end{pmatrix},
$$
and on the other hand,
$$
V_2 V_1 = \begin{pmatrix}
  X   &  0   & 0 & \\
  Y   &  W_2 & Z & \\
  R   &  0   & T & \\
\end{pmatrix}
\begin{pmatrix}
  A &  0 & 0   & \\
  0 &  I & 0   & \\
  D &  0 & W_1 & \\
\end{pmatrix}
 =
\begin{pmatrix}
  XA         &  0     & 0    & \\
  YA + ZD    &  W_2   & Z W_1& \\
  RA + TD    &  0     & T W_1& \\
\end{pmatrix}.
$$
Since $V_1$ and $V_2$ commute, we see that $Z W_1 = Z$, or, equivalently,
\begin{equation}\label{eq:Zzero}
(W_1^* - I)Z^* = 0 .
\end{equation}
But $W_1$ is a contraction (in fact, an isometry), and it has no invariant vectors (becuase it is the restriction of $V_1$ to $L_2 = \left( \tilde{L}_1 \vee \tilde{L}_2 \right) \ominus \tilde{L}_1$), thus by Proposition I.3.1 in \cite{SzNF70}, $W_1^*$ has no invariant vectors. This means that $W_1^* - I$ is injective, so (\ref{eq:Zzero}) implies that $Z^* = 0$, hence $Z=0$. It now follows that $Y = YA$, and arguing as in the previous sentence ($A^* A + D^* D = I$ implies that $A$ is a contraction, and it has no invariant vectors) we conclude that $Y=0$.

It turns out that $(\tilde{K},\tilde{V}_1,\tilde{V}_2)$, where $\tilde{K} := H \oplus M \oplus L_2$ and $\tilde{V}_i := V_i \big |_{\tilde{K}}$, is already an isometric dilation of $(T_1,T_2,H)$. This dilation has the property that $1$ is not an eigenvalue of $\tilde{V}_1$, and that $L := \{x \in \tilde{K} : \tilde{V}_2 x = x \} \subseteq H^\perp$. The same kind of reasoning demonstrated above can be used to show that $\tilde{V}_1$ and $\tilde{V}_2$ have the structures
$$
\tilde{V}_1 =
\begin{pmatrix}
  U_1    & 0 & \\
  0            & W & \\
\end{pmatrix},
$$
and
$$
\tilde{V}_2 =
\begin{pmatrix}
  U_2    & 0 & \\
  0            & I & \\
\end{pmatrix}.
$$
This shows that $(G,U_1,U_2)$, where $G := \tilde{K} \ominus L$ and $U_i := \tilde{V}_i \big |_G$, is also an isometric dilation of $(T_1,T_2,H)$, and this dilation has the property that $1$ is not an eigenvalue of neither $U_1$ nor $U_2$.
\end{proof}

Now, we turn to the proof of Theorem \ref{thm:mainthm}.

\begin{proof}
Let $T_1$ and $T_2$ be the cogenerators of the continuous semigroups $\{T_1 (s)\}_{s\geq0}$ and $\{T_2 (s)\}_{s\geq0}$, respectively (section III.8, \cite{SzNF70}). By Theorem III.8.1, \cite{SzNF70},
$$T_1 = \lim_{s\rightarrow 0+} \varphi_s (T_1 (s)),$$
and
$$T_2 = \lim_{s\rightarrow 0+} \varphi_s (T_2 (s)),$$
where
$$\varphi_s (x) = \frac{x-1+s}{x-1-s}\,,$$
and the limit is to be understood in the strong operator topology. For any $s>0$, $\varphi_s$ is analytic in some disc containing the spectrum of $T_i (s)$, $i=1,2$, so for all $s,t > 0$, $\varphi_s (T_1 (s))$ and $\varphi_t (T_2 (t))$ are limits in the operator norm of univariable polynomials in $T_1 (s)$ and in $T_2 (t)$, respectively. This means that for all $s,t >0$, $\varphi_s (T_1 (s))$ and $\varphi_t (T_2 (t))$ commute. Take a sequence $\{t_n\}$ of positive numbers decreasing to zero. By the proof of Theorem III.8.1, \cite{SzNF70}, it follows that $\{\varphi_{t_n} (T_1 (t_n))\}$ and $\{\varphi_{t_n} (T_2 (t_n))\}$ are bounded sequences, so
$$T_1 T_2 = \lim_{n \rightarrow \infty} \varphi_{t_n} (T_1 (t_n)) \varphi_{t_n} (T_2 (t_n)) = \lim_{n \rightarrow \infty} \varphi_{t_n} (T_2 (t_n)) \varphi_{t_n} (T_1 (t_n)) = T_2 T_1 ,$$
that is, $T_1$ and $T_2$ commute.

By Theorem III.8.1, \cite{SzNF70}, (again), $1$ is not an eigenvalue of $T_1$ or $T_2$. We apply Lemma \ref{lem:evalnot1} to obtain an isometric dilation $(V_1, V_2, K)$ of $(T_1, T_2, H)$ with $V_1$ and $V_2$ having no invariant vectors. Having no invariant vectors, $V_1$ and $V_2$ are themselves cogenerators of two continuous semigroups $\{V_1 (s) \}_{s \geq 0}$ and $\{V_2 (s) \}_{s \geq 0}$. According to Proposition III.9.2, \cite{SzNF70}, these semigroups consist of isometries.

It remains to be shown that for all $s,t > 0$, $V_1(s) V_2 (t) = V_2 (t) V_1 (s)$ and that $\{V_1(s) V_2 (t)\}_{s,t \geq 0}$ dilates $\{T_1(s) T_2 (t)\}_{s,t \geq 0}$. By Theorem III.8.1, \cite{SzNF70}, (one last time), $V_1$ and $V_2$ determine $\{V_1 (s) \}_{s \geq 0}$ and $\{V_2 (s) \}_{s \geq 0}$ by the relations
$$V_1 (s) = e_s (V_1) ,$$
and
$$V_2 (s) = e_s (V_2) ,$$
where
$$e_s (x) = \exp\left( s \frac{x+1}{x-1} \right) .$$
Now, $e_s (V_i)$, $i=1,2$, is defined as the strong limit $r \rightarrow 1^-$ of operators $e_{s,r} (V_i)$, where
$$e_{s,r}(x) = e_s (rx),$$
(section III.8, \cite{SzNF70}). The inequality $e_s(x) \leq 1$ for $x$ in the open unit disc, together with von Neumann's inequality, implies that $\|e_{s,r} (V_i)\| \leq 1$, so it follows that $V_1 (s)$ and $V_2 (t)$ commute for all $s,t \geq 0$.

Finally, to see that $\{V_1(s) V_2 (t)\}_{s,t \geq 0}$ dilates $\{T_1(s) T_2 (t)\}_{s,t \geq 0}$, we repeat the arguments that lead to (g) in Theorem III.2.3, \cite{SzNF70}. From the equality
$$T_1^m T_2^n = P_H V_1^m V_2^n \big |_H \,\,,\,\, m,n \in \mathbb{N},$$
it follows that
$$p(T_1) q(T_2) = P_H p(V_1) q(V_2) \big |_H ,$$
for all polynomials $p$ and $q$. Thus
$$e_{s,r}(T_1) e_{t,r}(T_2) = P_H e_{s,r}(V_1) e_{t,r}(V_2) \big |_H ,$$
for all $s,t \geq 0$ and $r \in (0,1)$. Because all the strong limits taken involve bounded families of operators, we arrive at
$$e_{s}(T_1) e_{t}(T_2) = P_H e_{s}(V_1) e_{t}(V_2) \big |_H ,$$
for all $s,t \geq 0$, which is just another way of writing equation (\ref{eq:what we want}). If we replace $K$ by $\bigvee_{s,t\geq0} V_1 (s) V_2 (t) H$ and $\{V_1 (s) \}, \{V_2 (s) \}$ by their restrictions to this space, we obtain a minimal dilation.
\end{proof}

\section{Regular unitary dilations for certain semigroups}

\subsection{Definitions and notation}

Throughout, $\Rp$ will denote the set $[0,\infty)$, $\Omega$ will denote some fixed set, $\mathbb{R}^\Omega$ will denote the additive group of all real valued functions on $\Omega$, and $\Rp^\Omega$ will denote the additive semigroup of all non-negative functions on $\Omega$. Addition and multiplication on $\mathbb{R}^\Omega$ are defined pointwise.

$\mathbb{R}^\Omega$ becomes a partially ordered set if one introduces the relation
$$s \leq t \Longleftrightarrow s(j) \leq t(j) \,\, , \,\, j\in\Omega .$$
The symbols $<$, $\ngeq$, etc., are to be interpreted in the obvious way.

A \emph{commensurable semigroup} is a semigroup $\Sigma$ such that for every $N$
elements $s_1, \ldots, s_N \in \Sigma$, there exist $s_0 \in \Sigma$ and $a_1, \ldots, a_N \in \mathbb{N}$ such that
$s_i = a_i s_0$ for all $i = 1, \ldots N$. For example, $\mathbb{N}$ is a commensurable semigroup. If $r\in \Rp$, then $r\cdot \mathbb{Q}_+$ is commensurable, and any commensurable subsemigroup of $\Rp$ is contained in such a semigroup.
Throughout this section $\cS$ will denote the semigroup
$$\cS = \sum_{i\in\Omega}\cS_i ,$$
where $\cS_i$ is a commensurable and unital (i.e., contains $0$) subsemigroup of $\Rp$.
To be more precise, $\cS$ is the subsemigroup of $\Rp^\Omega$ of finitely supported functions $s$ such that $s(j) \in \cS_j$ for all $j \in \Omega$. Still another way to describe $\cS$ is the following:
\bes
\cS = \left\{ \sum_{j\in \Omega} {\bf e_j}(s_j) : s_j \in \cS_j, {\rm \,\,all \,\, but\,\, finitely\,\, many\,} s_j {\rm's \,\,are\, }0 \right\},
\ees
where ${\bf e_i}$ is the inclusion of $\cS_i$ into $\prod_{j\in \Omega}\cS_j$. Here is a good example to keep in mind: if $|\Omega| = k \in \mathbb{N}$, and if $\cS_i = \mathbb{N}$ for all $i\in\Omega$, then $\cS = \mathbb{N}^k$.
We denote by $\cS - \cS$ the subgroup of $\mathbb{R}^\Omega$ generated by $\cS$. For $s \in \mathbb{R}^\Omega$ we shall denote by $s_+$ the element in $\cS$ that sends $j\in \Omega$ to $\max\{0,s(j)\}$, and $s_- = s_+ - s$.

If $u = \{u_1, \ldots, u_N\} \subseteq \Omega$, we let $|u|$ denote the number of elements in $u$ (this notation will only be used for finite sets). We shall denote by ${\bf e}[u]$ the element of $\mathbb{R}^\Omega$ having $1$ in the $i$th place for every $i\in u$, and having $0$'s elsewhere, and we denote $s[u]: = {\bf e}[u]\cdot s$, where multiplication is pointwise. If $u = \{j\}$ we shall write $s[j]$ instead of $s[\{j\}]$.

\begin{definition}
Let $\Sigma$ be a subsemigroup of $\Rp^\Omega$. We say that $\Sigma$ is \emph{$_{+,-}$-closed} if for every 
$s \in \Sigma$ both $s_+$ and $s_-$ are in $\Sigma$.
\end{definition}
Note that $\cS$ is a $_{+,-}$-closed subsemigroup of $\Rp^\Omega$, and so is $\Rp^\Omega$.
\begin{definition}
Let $\Sigma$ be a $_{+,-}$-closed subsemigroup of $\Rp^\Omega$, and let $T = \{T_s\}_{s\in\Sigma}$ be a semigroup of contractions on a Hilbert space $H$. A family of unitaries (isometries) $U = \{U_s\}_{s\in\Sigma}$ acting on a Hilbert space $K \supseteq H$ is said to be a \emph{regular unitary (isometric) dilation of $T$} if for all $s \in \Sigma - \Sigma$
$$P_H U_{s_-}^*U_{s_+}\big|_H = T_{s_-}^*T_{s_+} .$$
\end{definition}
\begin{definition}
Let  $T = \{T_s\}_{s\in\cS}$ be a semigroup of contractions on a Hilbert space $H$. $T$ is said to \emph{doubly commute} if for all $i,j\in \Omega$ such that $i \neq j$, and all $s_i \in \cS_i, s_j \in \cS_j$, the following equation holds:
$$T_{{\bf e_j}(s_j)}T_{{\bf e_i}(s_i)}^* = T_{{\bf e_i}(s_i)}^* T_{{\bf e_j}(s_j)} .$$
\end{definition}

\subsection{Regular isometric and regular unitary dilations}
\begin{lemma}\label{lem:isodil}
Let $\Sigma$ be a $_{+,-}$-closed subsemigroup of $\Rp^\Omega$, and let $V = \{V_s\}_{s\in \Sigma}$ be a semigroup of isometries on a Hilbert space $K$. Then $V$ has a unique, minimal, regular unitary dilation. If $\Sigma$ has a topology, and $V$ is weakly continuous, then $U$ is strongly continuous.
\end{lemma}
\begin{proof}
By Theorem I.9.2, \cite{SzNF70}, there exists a commutative system
$U = \{U_s\}_{s\in\cS}$ on a Hilbert space $L\supseteq K$ which is
a minimal, regular dilation of the family $V$, that is, for all
$s,t \in \Sigma$,
$$P_K U_s^* U_t \big|_K = V_s^* V_t .$$
Since a unitary dilation of an isometry must be an extension, we have that for all $k \in K$ and all $s,t \in \Sigma$,
$$P_K U_s U_t k = P_K U_s V_t k = P_K V_{s+t} k = V_{s+t} k .$$
In other words, both $U_s U_t$ and $U_{s+t}$ are unitary dilations of $V_{s+t}$. From the uniqueness of the minimal regular unitary dilation
we get that $U_s U_t = U_{t+s}$, that is, $U$ is a semigroup.

Continuity and uniqueness follow as in the proofs of Theorems
I.7.1 and I.9.1 of \cite{SzNF70}, respectively.
\end{proof}

The following proposition will allow us to restrict our attention only to unitary dilations.
\begin{proposition}\label{prop:unitiso}
Let $T = \{T_s\}_{s\in\cS}$ be a contractive semigroup on a Hilbert space $H$. The existence of a (minimal)[regular] unitary dilation for $T$ is equivalent to the existence of a (minimal)[regular] isometric dilation for $T$.
\end{proposition}
\begin{proof}
First, note that the existnce of a unitary/isometric dilation is
equivalent to the existence of a minimal unitary/isometric
dilation. Next, let $V = \{V_s\}_{s \in \cS}$ is an isometric
dilation of $T$ on a Hilbert space $K$. By Lemma \ref{lem:isodil},
there is a regular unitary dilation $U$ of $V$. Then $U$ is a
Unitary dilation of $T$, and if $V$ is a regular dilation of $T$,
then so is $U$. This shows that the existence of a [regular]
isometric dilation for $T$ implies the existence of a [regular]
unitary dilation for $T$.

To show the converse, assume that $U = \{U_s\}_{s\in\cS}$ is a minimal, regular, unitary dilation of $T$ on $L \supseteq H$. Minimality means that
$$L = \bigvee_{s\in \cS - \cS}U_s H .$$
Define
$$K := \bigvee_{s\in \cS}U_s H ,$$
and $V_s := U_s\big|_K$. $V = \{V_s\}_{s\in \cS}$ is clearly a family of isometries. For all $k\in K$, $s,t \in \cS$,
$$V_{s}V_{t}k = U_{s}U_{s}k = U_{s+t}k = V_{s+t}k ,$$
so $V$ is a semigroup.
For all $h,g\in H$, $s \in \cS$,
$$\langle V_{s_-}^* V_{s_+} h, g\rangle = \langle U_{s_+} h, U_{s_-}g \rangle = \langle T_{s_-}^* T_{s_+} h ,g \rangle ,$$
so $V$ is a regular dilation of $T$. The case that $U$ is not necessarily a regular dilation is similar.
\end{proof}

\subsection{Regular unitary dilations of semigroups over $\cS$}
\begin{theorem}
Let $T=\{T_s\}_{s\in\cS}$ be a semigroup of contractions on a Hilbert space $H$. $T$ has a regular unitary dilation
if and only if for every finite set $v \subseteq \Omega$ and every $s \in \cS$,
\be\label{eq:condition}
\sum_{u\subseteq v}(-1)^{|u|}T^*_{s[u]}T_{s[u]} \geq 0 .
\ee
The minimal regular unitary dilation of $T$ is unique up tp unitary equivalence. If $\cS$ has a topology on it, and if $T$ is weakly continuous, then the unitary dilation is strongly continuous.
\end{theorem}
\begin{proof}
Assume that $T$ has a regular unitary dilation $V$. Let $v \subseteq \Omega$ be finite, and let $s\in \Omega$.
The set of contractions
$\{T_{s[j]}\}_{j\in v}$ has a regular unitary dilation, namely $\{V_{s[j]}\}_{j\in v}$. By Theorem I.9.1, \cite{SzNF70}, (\ref{eq:condition}) holds.

Assume that for all finite $v \subseteq \Omega$ and $s\in \cS$, (\ref{eq:condition}) holds. We define a function
$T(\cdot): \cS - \cS \rightarrow B(H)$ by
$$\hat{T}(s) = T_{s_-}^*T_{s_+} .$$
By the commensurablity of the $\cS_i$'s and by the proof of
Theorem I.9.1 in \cite{SzNF70}, it follows that $\hat{T}(\cdot)$
is a positive definite function. By Theorem I.7.1 in
\cite{SzNF70}, there is a Hilbert space $K\supseteq H$ and a
unitary representation $U:\cS - \cS \rightarrow B(K)$ such that
$$P_H U(s_-)^*U(s_+)\big|_H = P_H U(s)\big|_H = \hat{T}(s) = T_{s_-}^*T_{s_+} \,\, , \,\, s\in \cS - \cS .$$
Uniqueness and continuity are proved in the usual way (I.9.1, \cite{SzNF70}).
\end{proof}
\begin{corollary}
Let $T = \{T_s\}_{s\in\cS}$ be a doubly commuting contractive semigroup on a Hilbert space $H$. Then
$T$ has a unique, minimal, doubly commuting, regular unitary/isometric dilation.
\end{corollary}
\begin{proof}
A doubly commuting semigroup must satisfy (\ref{eq:condition}), thus $T$ has a unique, minimal, regular unitary/isometric dilation. Every unitary semigroup doubly commutes. To see that a minimal, regular, isometric dilation of a doubly commuting semigroup doubly commutes, one uses the commensurability of the $\cS_i$'s and
Theorem 3.10 of \cite{S06} (the content of that Theorem, simplified to this context, is that a minimal, regular, isometric dilation of a doubly commuting semigroup \emph{over} $\mathbb{N}^k$ doubly commutes).
\end{proof}

\subsection{Unitary dilation of a multi-parameter (not necessarily continuous) semigroup of coisometries}
\begin{theorem}
Let $\Sigma$ be a $_{+,-}$-closed subsemigroup of $\Rp^\Omega$, and let $T = \{T_s\}_{s\in\Sigma}$ be a subsemigroup of coisometries on a Hilbert space $H$. There exists a Hilbert space $K \supseteq H$ and a family of $U = \{U_s\}_{s\in\Sigma}$ unitaries on $K$ such that:
\begin{enumerate}
    \item For all $s\in\Sigma - \Sigma$,
    $$P_H U_{s_-}U_{s_+}^*\big|_H = T_{s_-}T_{s_+}^* ;$$
    \item $K = \bigvee_{s\in \Sigma} U_{s} H$.
\end{enumerate}
If $\Sigma$ has a topolgy on it and $T$ is weakly continuous, then $U$ is strongly continuous. Furthermore, a unitary dilation satisfying the assertions of the theorem is unique up to unitary equivalence.
\end{theorem}
\begin{proof}
The theorem follows from Lemma \ref{lem:isodil} by looking at $T^* = \{T_s^*\}_{s\in\Sigma}$. To justify item (2), note that $U^*$ is an extension of $T^*$.
\end{proof}


\begin{thebibliography}{9}

%

\bibitem {Ptak} M. Ptak, \emph{Unitary dilations of multiparameter semigroups of operators},
 Ann. Polon. Math.  \textbf{45}  (1985),  no. 3, 237--243.

\bibitem {Slocinski} M. S\l oci\'nski, \emph{Unitary dilation of two-parameter semi-groups of contractions},
Bull. Acad. Polon. Sci. Se'r. Sci. Math. Astronom. Phys.
\textbf{22} (1974), 1011--1014.

\bibitem {S06} Baruch Solel, \emph{Regular dilations of representations of product systems},
arXiv:math.OA/0504129, 2006.

\bibitem {SzNF70}B. Sz.-Nagy and C. Foia\c{s}, \emph{Harmonic Analysis of
Operators in Hilbert Space}, North-Holland, Amsterdam, 1970.


\end{thebibliography}
\end{document}